\documentclass[11pt]{amsart}
\usepackage[utf8]{inputenc}
\usepackage{amsthm}
\usepackage{amsmath}
\usepackage{amssymb}
\usepackage{amscd}
\usepackage{hyperref}
\usepackage{cleveref}
\usepackage{relsize}
\usepackage{tikz-cd}
\usepackage{epsdice}
\usepackage[left=1.2in,top=1.2in,bottom=1.2in,right=1.2in]{geometry}
\theoremstyle{definition}
\newtheorem{theorem}{Theorem}[section]
\newtheorem{lemma}[theorem]{Lemma}
\newtheorem{proposition}[theorem]{Proposition}
\newtheorem{remark}[theorem]{Remark}
\newtheorem{definition}[theorem]{Definition}
\newtheorem{eg}[theorem]{Example}
\newtheorem{conjecture}[theorem]{Conjecture}
\newtheorem{corollary}[theorem]{Corollary}

\newtheorem{question}[theorem]{Question}

\newtheorem*{theorem*}{Theorem}
\newtheorem*{proposition*}{Proposition}
\usepackage{mathtools}
\usepackage{mathrsfs}
\usepackage{enumitem}
\hypersetup{
    colorlinks=true,
    linkcolor=cyan,
    filecolor=magenta,
    citecolor= cyan,
    urlcolor=blue}

\usepackage [backend= biber, style= alphabetic, sorting= nty]{biblatex}
\begin{document}

\title{IND-\'ETALE VS FORMALLY \'ETALE}

\author{Shubhodip Mondal}
\address{Max-Planck-Institut f\"ur Mathematik, Vivatsgasse 7, 53111 Bonn, Germany }
\email{mondal@mpim-bonn.mpg.de}
\author{Alapan Mukhopadhyay}
\address{Department of Mathematics, University of Michigan, 530 Church Street,
  Ann Arbor, MI 48109}
\email{alapanm@umich.edu}

\date{}

\maketitle
\begin{abstract}
    We show that when $A$ is a reduced algebra over a characteristic zero field $k$ and the module of K\"ahler differentials $\Omega_{A/k}=0$, then $A$ is ind-\'etale, partially answering a question of Bhatt. As further applications of this result, we deduce a rigidity property of Hochschild homology and special instances of Weibel's conjecture \cite{weibel} and Vorst's conjecture \cite{vorst} without any noetherian assumptions. 
    
\end{abstract}

\tableofcontents
\section{Introduction}
In this article, all rings are commutative and unital, unless otherwise mentioned.

Fix a field $k$. Recall that a finite type $k$-algebra $R$ is \textit{\'etale over $k$} if the module of K\"ahler differentials $\Omega_{R/k}$ is zero.
\begin{definition}
A $k$-algebra is said to be \textit{ind-\'etale} if as a $k$-algebra, it is isomorphic to a direct limit of some direct system of \'etale algebras over $k$.
\end{definition}

If $R$ is an \'etale algebra, then the cotangent complex $\mathbb{L}_{R/k}$ is exact, i.e., $H^i (\mathbb{L}_{R/k})=0$ for all $i \in \mathbb Z.$ We refer to \cite[Tag 08P5]{Stacksproject} for the definition and basic properties of the cotangent complex. When $A$ is a smooth algebra over $k$, then the cotangent complex agrees with the module of K\"ahler differentials. But in general, the cotangent complex is a complex of $A$-modules or more naturally, an object in the derived category $D(A)$ of chain complexes over A. Since the formation of cotangent complex commutes with taking direct limits, it follows that for an ind-\'etale algebra $A$, the cotangent complex $\mathbb{L}_{A/k}$ is exact.

In \cite{BhOriginal}, Bhargav Bhatt raised the following question asking whether conversely exactness of $\mathbb{L}_{A/k}$ -- i.e., $A$ being \textit{formally \'etale} implies ind-\'etaleness of $A$.
\begin{question}[Bhatt]\label{main question}
Let $k$ be a field of characteristic zero. Does there exist a $k$-algebra $A$, such that the cotangent complex $\mathbb{L}_{A/k}$ is exact, yet $A$ is not ind-\'etale over $k$? \textbf{\textendash}  see \cite[Question 0.3]{BhOriginal} and \cite[Question C.3]{Morrow}.
\end{question}
Note that ind-\'etale algebras are necessarily reduced. In this note, we answer the question above when $A$ is additionally assumed to be reduced.
\vspace{2mm}

\begin{theorem}\label{t: formally etale implies ind-etale in the introduction}
(see \Cref{formally unramified plus reduced implies ind-etale})
Let $k$ be a field of characteristic zero and $A$ be a reduced $k$-algebra -- not assumed to be noetherian. If $\Omega_{A/k}=0$, then $A$ is ind-\'etale.     
\end{theorem}
Since the module of K\"ahler differentials is the zeroth cohomology of the cotangent complex, the exactness of 
$\mathbb{L}_{A/k}$ implies $\Omega_{A/k}=0$. So our theorem partially answers \Cref{main question} by showing that when the algebra $A$ is known to be reduced, then the much \textit{weaker} assumption of $\Omega_{A/k}=0$, implies that $A$ is ind-\'etale. Thus, \Cref{t: formally etale implies ind-etale in the introduction} now reduces \Cref{main question} to the following question:
\begin{question}
Let $k$ be a field of characteristic zero. Does there exist a $k$-algebra $A$ such that the cotangent complex $\mathbb{L}_{A/k}$ is exact, yet $A$ is not reduced?
\end{question}

\begin{remark}Note that, merely assuming $\Omega_{A/k}=0$ and $k$ is a field of characteristic zero does \textit{not} imply reducedness of $A$, as is shown by an example originally due to Ofer Gabber. See \Cref{Gabber's counterexample} and \cite[Theorem 2.2]{Alapan}.
\end{remark}
The main difficulty in proving \Cref{t: formally etale implies ind-etale in the introduction} is in dealing with the lack of any finiteness or noetherian assumptions. The key new ingredient in the proof is the observation of the following general result, whose proof, in turn makes judicious use of localization constructions and minimality tricks to circumvent issues caused by the lack of noetherian assumption.

\begin{proposition}[see \Cref{unramified implies integral}]\label{propb}
Let $k$ be a field of characteristic zero and $A$ be a $k$-algebra -- not assumed to be noetherian. Suppose that there is a $k$-algebra injection from the polynomial ring, $\iota: k[t_1, \ldots, t_s] \hookrightarrow A$ for some $s \geq 1$. Then $d\iota(t_1) \wedge d\iota(t_2) \wedge \ldots \wedge d\iota(t_s)$ is a nonzero element of $\wedge^s_A \Omega_{A/k}$.
\end{proposition}
When phrased in terms of \textit{transcendence cardinality} introduced in \Cref{d:trans cardinality}, \Cref{propb} implies that vanishing of $\wedge^s_A \Omega_{A/k}$ forces the $k$-transcendence cardinality of $A$ to be at most $s-1$.

\begin{remark}

The analogue of \Cref{main question} in positive characteristics has a negative answer. There is an example due to Bhatt of a positive characteristic field $k$ and a $k$-algebra $A$ such that the cotangent complex $\mathbb{L}_{A/k}$ is exact, but $A$ is not reduced, thus not ind-\'etale -- see \cite[Proposition 0.2]{BhOriginal}, where the example is attributed to Ofer Gabber. 
\end{remark}

We will use \Cref{t: formally etale implies ind-etale in the introduction} to prove the following result about vanishing of Hochschild homology. We point out that in \cite{10.1155/S1073792892000035}, certain vanishing of Hochschild homology $\mathrm{HH}_i(A)$ for an algebra $A$ had been used to give a criteria of smoothness.

\begin{proposition}[see \Cref{hhhh}]\label{HH_1}Let $k$ be a field of characteristic $0$. Let $A$ be a reduced commutative $k$ algebra -- not assumed to be noetherian. If $\mathrm{HH}_1(A/k)= 0,$ then $\mathrm{HH}_i(A/k)= 0$ for all $i \ge 1.$

\end{proposition}{}

In \Cref{HH}, we briefly review the definition of Hochschild homology and prove the above proposition. The proof is manifestly based on techniques from commutative algebra and we point out that the commutativity assumption in \Cref{HH_1} is very sharp. \Cref{exampleBen} shows that a similar assertion is false without the commutativity assumption; we learnt this example from Antieau.

\vspace{2mm}
\Cref{propb} also has some unexpected consequences. It can be used to deduce new special instances of a question of Weibel \cite[Question 2.9]{weibel} and Vorst's conjecture \cite{vorst} without any finiteness or noetherian assumptions. The applications in this direction were pointed out by Morrow after we shared a draft version of our paper with him; we thank him heartily for generously sharing his observations with us. These questions belong to the area of $K$-regularity, which roughly speaking, uses algebraic $K$-theory to study regularity of commutative rings. We briefly recall some necessary definitions and results in $K$-theory in \Cref{Kreg} and include the new applications. In what follows, \Cref{intro1} is related to the question of Weibel and \Cref{intro2} is a non-noetherian case of Vorst's conjecture.

\begin{proposition}[see \Cref{weib}]\label{intro1}
Let $A$ be a commutative $k$-algebra over a field $k$ of characteristic zero such that the module of $d+1$-forms $\wedge^{d+1}\Omega_{A/k}=0.$ Then 

\begin{enumerate}
    \item For  $n<-d$, the $K$-groups $K_n(A) = 0$.
    \item $A$ is $K_n$-regular (see \Cref{teach}) for all $n \le -d.$ 
\end{enumerate}

\end{proposition}

\begin{proposition}[see \Cref{vorstcon}]\label{intro2}
Let $A$ be a commutative $k$-algebra over a field $k$ of characteristic zero such that the module of K\"ahler differentials $\Omega_{A/k}=0.$ If $A$ is $K_1$-regular, then $A$ is ind-\'etale.  Moreover, $A$ is $K_n$-regular for all integers $n.$
\end{proposition}{}

\subsection*{Acknowledgements} We are very grateful to Matthew Morrow for several helpful exchanges, pointing out the connections with $K$-theory, as well as for bringing Bhatt's question to our attention during the Arizona Winter School in 2018. We are also very thankful to Ben Antieau, Bhagav Bhatt, Mel Hochster and Karen Smith for helpful conversations.  The first named author thanks the support of NSF
grant DMS \#1801689, NSF FRG grant \#1952399 and the Rackham international student fellowship; the second named author thanks the support of NSF grants DMS \#2101075, \#1801697, NSF FRG grant \#1952399 and Rackham one term dissertation fellowship while working on this article.

\section{Main results}
Fix a field $k$. In this section, we first define and explore the notion of \textit{$k$-transcendence cardinality} of a $k$-algebra. Transcendence cardinality shows up later in \Cref{weib}. The main result of this section, \Cref{unramified implies integral}, provides a criteria for finite transcendence cardinality in terms of the vanishing of module of differential forms.

For the definition of \textit{cardinal numbers} appearing in the definition below refer to \cite[Chapter III]{Bourbakiset}. 
\begin{definition}\label{d:trans cardinality}
Let $A$ be a $k$-algebra. Given a cardinal number $J$, we say that $A$ has $k$-transcendence cardinality at least $J$, if there is a set $S$ of cardinality $J$ and a $k$-algebra injection $k[\{t_j \, | \, j \in S\}] \hookrightarrow A$. The \textit{$k$-transcendence cardinality} of $A$ is $\text{Sup}\{J\, \, | \,\, A \,\, \text{has} \,\, k-\text{transcendence cardinality at least} \,\, J\}$. Here the supremum is taken in the set of all cardinal numbers originating from subsets of $A$ and the supremum is also a cardinal number- see \Cref{r: cardinal number}.

When the $k$-transcendence cardinality of $A$ is a natural number, we call that natural number the \textit{$k$-transcendence degree} of $A$.\\
\end{definition}
\begin{remark}\label{r: cardinal number}
 The cardinal numbers of subsets of a given set form a well ordered set- see Theorem 1, \cite{cardinal}. Hence the supremum in \Cref{d:trans cardinality} is again a cardinal number.
\end{remark}
We establish some basic properties of the notion of transcendence cardinality listed below.
\begin{proposition}\label{p: transcendence degree}Let $A$ be a $k$-algebra, not necessarily noetherian.
\begin{enumerate}
    \item If $A$ is finite type over $k$, then the $k$-transcendence cardinality of $A$ is finite and is the same as the Krull dimension of $A$.
    \item  If $A$ has $k$-transcendence cardinality at most $n \in \mathbb{N}$, then every finite type $k$-subalgebra of $A$ has Krull dimension at most $n$. Moreover the Krull dimension of $A$ is at most $n$.
    \item Let $B$ be a $k$-algebra such that $B$ contains $A$ and is module finite over $A$. Then the transcendence cardinality of $B$ is finite if and only if the transcendence cardinality of $A$ is finite. When both the transcendence cardinalities are finite, those are the same.
    \item Let $ \phi: A \rightarrow C$ be a finite $k$-algebra homomorphism. If the $k$-transcendence cardinality of $A$ is $n \in \mathbb N$, then the $k$-transcendence cardinality of $C$ is at most $n$.
    \item The $k$-transcendence cardinalities of $A$ and $A_{\text{red}}$ are the same.
    \item Let $A$ be a domain with finite $k$-transcendence cardinality. The transcendence cardinality of the fraction field of $A$ is the same as that of $A$.
    \item If $k$ has characteristic zero and $\wedge^s \Omega_{A/k}=0$, then $A$ has $k$-transcendence cardinality at most $s-1$.
\end{enumerate}
\end{proposition}
\begin{remark}
The Krull dimension can be much lower than the $k$-transcendence cardinality. For example, a field extension $L$ of $k$ can have arbitrarily large $k$-transcendence cardinality, while the Krull dimension of $L$ is zero.
\end{remark}
\begin{proof}[Proof of \Cref{p: transcendence degree}] (1) Let $d$ be the Krull dimension of $A$. Noether normalization (see \cite[00OY]{Stacksproject}) guarantees a module finite inclusion $k[x_1, \ldots, x_d] \hookrightarrow A$. So the $k$-transcendence cardinality of $A$ is at least $d$. We show that the $k$-transcendence cardinality of $A$ is at most $d$. For that, we use the next lemma to reduce the problem to the case where $A$ is a domain.
\begin{lemma}\label{a multiplicative set avoids a minimal prime}
Let $S$ be a multiplicative set in a commutative ring $R$, such that $0 \notin S$. There is a minimal prime $\mathfrak{p}$ of $R$ such that $S \subseteq R-\mathfrak{p}$.
\end{lemma}
\begin{proof}
Since $0 \notin S$, the localization $S^{-1}R$ is nonzero. So $S^{-1}R$ has a minimal prime $\mathfrak{q}$. We can take $\mathfrak{p}$ to be the contraction of $\mathfrak{q}$ via the natural map $R \rightarrow S^{-1}R$.
\end{proof}
Now given a $k$-algebra inclusion $\phi: k[x_1, \ldots, x_n] \hookrightarrow A$, take $S= \phi(k[x_1, \ldots, x_n] \setminus 0)$. Using \Cref{a multiplicative set avoids a minimal prime}, choose a minimal prime $\mathfrak{p}$ of $A$ such that $S \cap \mathfrak{p}= \emptyset$. This means the composition of $\phi$ with the quotient $A \rightarrow \frac{A}{\mathfrak p}$ is also injective. The last injection gives an injection of the fraction fields $k(x_1, \ldots, x_n) \hookrightarrow \text{Frac}(\frac{A}{\mathfrak{p}})$. So $n$ is at most the $k$-transcendence degree of $\text{Frac}(\frac{A}{\mathfrak{p}})$. Since the $k$-transcendence degree of $\text{Frac}(\frac{A}{\mathfrak{p}})$ is the Krull dimension of $\frac{A}{\mathfrak{p}}$ and the Krull dimension of $\frac{A}{\mathfrak{p}}$ is at most $d$ (\cite[Prop. 14]{Serre}), $n \leq d$.\\\\
(2) The Krull dimension of any finite type $k$-subalgebra $B$ of $A$ is at most the $k$-transcendence degree of $B$. Since the $k$-transcendence degree of $B$ is at most $n$, we are done.\\

We now prove that the Krull dimension of $A$ is at most $n$.  Let $\mathfrak{p_0} \subseteq \mathfrak{p_1} \subseteq \ldots \subseteq \mathfrak{p_m}$ be chain of prime ideals of $A$- where each containment is strict. For each $j \geq 1$, choose $x_j \in \mathfrak{p_j} \setminus \mathfrak{p_{j-1}}.$ Let $B$ be the $k$-subalgebra of $A$ generated by $x_1, \ldots, x_m$. So we have a chain of prime ideals in $B$ with strict containments, 
$$\mathfrak{p_0} \cap B \subsetneq  \mathfrak{p_1} \cap B  \subsetneq \ldots \subsetneq \mathfrak{p_m} \cap B.$$
So $m$ is at most the Krull dimension of $B$. Since the Krull dimension of $B$ is at most $n$ by part $1$, the Krull dimension of $A$ is at most $n$.\\\\
(3) We show that if the $k$-transcendence cardinality of $A$ is $n \in \mathbb N$, then the $k$-transcendence cardinality of $B$ is also $n$. Given any $k$-algebra inclusion $\phi: k[x_1, \ldots, x_m] \hookrightarrow B$, we can choose finite type $k$-subalgebras $B' \subseteq B$ and $A' \subseteq A$, such that $\text{Im}(\phi) \subseteq B'$, $A' \subseteq B'$ and $A' \hookrightarrow B'$ is module finite. The choice can be made as follows: for each $j$, $1 \leq j \leq m$, there is a nonzero monic polynomial $F_j \in A[t]$ such that $F_j(\phi(x_j))=0$. Take $A'$ to be the $k$-subalgebra of $A$ generated by the coefficients of $F_j$'s where $j$ varies. Take $B'$ to be $A'$-subalgebra of $B$ generated by all the $\phi(x_j)$'s.

Now, by (1), the $k$-transcendence degree of $B'$ is the Krull dimension of $B'$. Since $A' \subseteq B'$ is module finite, the Krull dimension of $A'$ and $B'$ are the same. By (2), the Krull dimension of $A'$ is at most $n$. So the $k$-transcendence degree of $B'$ is at most $n$. Thus $m \leq n$, proving that the $k$-transcendence cardinality of $B$ is also at most $n$. Again since $A \subseteq B$, the $k$-transcendence cardinality of $B$ is at least $n$.

If the $k$-transcendence cardinality of $B$ is finite, the $k$- transcendence cardinality of $A$ is also finite as $A \subseteq B$. Moreover the $k$-transcendence cardinalities of $A$ and $B$ coincide by the argument above.\\\\
(4) Since there is a $k$-algebra surjection $A \rightarrow \phi(A)$, the $k$-transcendence cardinality of $\phi(A)$ is at most $n$. Since $\phi(A) \subseteq C$ is module finite, by (3), the $k$-transcendence cardinality of $C$ is at most that of $\phi(A)$ and the later is at most $n$.\\\\
(5) Given a set $S$ of cardinality $J$ and a $k$-algebra injection $\phi: k[\{t_j \, | \, j \in S\}] \hookrightarrow A$, the intersection of the image of $\phi$ and the nilradical of $A$ is zero. Thus composing $\phi$ with the surjection $A \rightarrow A_{\text{red}}$ also gives an injection. Thus the transcendence cardinality of $A_{\text{red}}$ is at least that of $A$. Given a set $S$ and a $k$-algebra injection $\psi: k[\{x_j \, | \, j \in S\}] \rightarrow A_{\text{red}}$, lift $\psi$ to a $k$-algebra map $k[\{x_j \, | \, j \in S\}] \rightarrow A$; the lift is necessarily injective. So the transcendence cardinality of $A_{\text{red}}$ is at most that of $A$.\\\\
(6) Suppose that the transcendence cardinality of $A$ is $n \in \mathbb{N}$. It is enough to show that the transcendence cardinality of $\text{Frac}(A)$ is at most $n$. By contradiction, assume that $\text{Frac}(A)$ contains elements $\frac{a_1}{b_1}, \frac{a_2}{b_2}, \ldots, \frac{a_{n+1}}{b_{n+1}}$, which are algebraically independent over $k$, where all $a_i, b_i$'s are in $A$. Then the subalgebra $k[a_1, \ldots, a_{n+1}, b_1, \ldots, b_{n+1}] \subseteq A$ has transcendence degree at least ${n+1}$ as its fraction field contains $k[\frac{a_1}{b_1}, \frac{a_2}{b_2}, \ldots, \frac{a_{n+1}}{b_{n+1}}]$. Since the transcendence degree of $A$ is $n$, we get a contradiction. \\\\
(7) This assertion follows from \Cref{unramified implies integral} proven below.

\end{proof}
\begin{proposition}\label{unramified implies integral}
Let $k$ be a field of characteristic zero and $A$ be a $k$-algebra -- not assumed to be noetherian. Suppose that there is a $k$-algebra injection from the polynomial ring, $\iota: k[t_1, \ldots, t_s] \hookrightarrow A$ for some $s \geq 1$. Then $d\iota(t_1) \wedge d\iota(t_2) \wedge \ldots \wedge d\iota(t_s)$ is a nonzero element of $\wedge^s_A \Omega_{A/k}$.
\end{proposition}
\begin{proof}
We first prove \Cref{unramified implies integral} assuming that $A$ is a field and then deduce the general case from the field case in a few steps.\\

Assume that $A$ is a field. Pick a subset $\{x_i \}_{i \in I}$ of $A$ such that $\{\iota(t_1), \ldots, \iota(t_s)\} \cup \{x_i \}_{i \in I}$ is a $k$-transcendence basis of $A$; for example, $\{x_i \}_{i \in I}$ can be chosen to be a $k(\iota(t_1), \ldots, \iota(t_s))$-transcendence basis of $A$; see \cite[Tag 030F]{Stacksproject}. Set $L$ to be smallest subfield of $A$ containing $k$ and  $\{\iota(t_1), \ldots, \iota(t_s)\} \cup \{x_i \}_{i \in I}$. For any finite field extension $L' \supseteq L$ where $L' \subseteq A$, since $L \subseteq L'$ is separable, we have an isomorphism,
\begin{equation}\label{isomorphism for the intermediate field}
    \Omega_{L/k}\otimes_L L' \cong \Omega_{L'/k} \ ;
\end{equation}
see \cite[Chapter 6, Lemma 1.13]{QingLiu}. Varying $L'$ over finite extensions of $L$ such that $L' \subseteq A$, we get a direct system of isomorphisms from \Cref{isomorphism for the intermediate field}; taking the direct limit of this direct system of isomorphisms we get an isomorphism
\begin{equation}\label{isomorphism of Kahler differentials}
\Omega_{L/k}\otimes_L A \cong \Omega_{A/k}.
\end{equation}
To get \Cref{isomorphism of Kahler differentials}, we have used that formation of modules of K\"ahler differentials commute with taking direct limit (see \cite[Tag 00RM]{Stacksproject}) and $A$ is the direct limit of the fields $L'$. Since $\Omega_{L/k}$ is isomorphic to the free $L$-module with basis $\{d\iota(t_1), \ldots, d\iota(t_s)\} \cup \{dx_i\}_{i \in I}$, \Cref{isomorphism of Kahler differentials} implies that $\Omega_{A/k}$ is a free $A$-module with basis $\{d\iota(t_1), \ldots, d\iota(t_s)\} \cup \{dx_i\}_{i \in I}$. Hence $d\iota(t_1) \wedge d\iota(t_2) \wedge \ldots \wedge d\iota(t_s)$ is a nonzero element of $\wedge^s_A \Omega_{A/k}$.\\

Given a $k$-algebra $A$ as in \Cref{unramified implies integral}, which is not necessarily a field, set $A'$ to be $A$ modulo the nilradical of $A$. Then the composition $k[t_1, \ldots, t_s] \xhookrightarrow{\iota} A \rightarrow A'$ is also injective; denote the composition by $\phi$. Set $S= \phi(k[t_1, \ldots, t_s] \setminus {0})$. We note that $S$ is a multiplicative set. Using \Cref{a multiplicative set avoids a minimal prime} we can choose a minimal prime $\mathfrak p$ of $A'$ such that $S \subseteq A'- \mathfrak{p}$. So the image of any nonzero element of $k[t_1, \ldots, t_s]$ under the composition $k[t_1, \ldots, t_s] \xrightarrow{\phi}A' \rightarrow A'_{\mathfrak{p}}$ is a unit; hence the composition is also injective. Denote the last composition by $\psi$. We have a commutative diagram,\\
\begin{equation}\label{commutative diagram}
\begin{tikzcd}
{\wedge^s_{k[t_1, \ldots, t_s]}\Omega_{k[t_1, \ldots, t_s]/k}} \arrow[rd, "\wedge^s d\psi"'] \arrow[rr, "\wedge^s d\iota"] &                                                         & \wedge^s_A\Omega_{A/k} \arrow[ld] \\
                                                                                                                        & \wedge^s_{A'_{\mathfrak{p}}}\Omega_{A'_{\mathfrak{p}}/k} &                                  
\end{tikzcd}
\end{equation}
where the unlabelled downward arrow is induced by the canonical map $A \rightarrow A'_{\mathfrak{p}}$. We want to show that $\wedge^s d\iota (dt_1 \wedge\ldots\wedge dt_s)$ is nonzero. To that end, first note that $A'_{\mathfrak{p}}$ is a field: since $\mathfrak{p}$ is minimal, the only prime ideal of $A'_{\mathfrak{p}}$ namely $\mathfrak{p}A'_{\mathfrak{p}}$ coincides with the nilradical of $A'_{\mathfrak{p}}$, which is zero as $A'$ and hence $A'_{\mathfrak{p}}$ is reduced. Now by the field case of \Cref{unramified implies integral}, $\wedge^s d\psi(dt_1\wedge\ldots \wedge dt_s)$ is nonzero. The commutativity of diagram \ref{commutative diagram} implies that $\wedge^s d\psi(dt_1\wedge\ldots \wedge dt_s)$ is the image of $\wedge^s d\iota (dt_1 \wedge\ldots\wedge dt_s)$. So $\wedge^sd\iota (dt_1 \wedge\ldots\wedge dt_s)$ must be a nonzero element of $\wedge^s_A\Omega_{A/k}$.  
\end{proof}
As an immediate corollary we get,
\begin{corollary}\label{vanishing Kahler differential bounds transcendence degree}
Let $k$ be a characteristic zero field and $A$ be a $k$-algebra -- not necessarily noetherian. If $\wedge^s_A \Omega_{A/k}=0$, then there cannot be a $k$-algebra injection $k[t_1, \ldots, t_s] \rightarrow A$.
\end{corollary}
\begin{remark}
\noindent
\begin{enumerate}

   \item The converse to \Cref{vanishing Kahler differential bounds transcendence degree} is false as the following example shows. For any characteristic zero field $k$, take $A= k[x]/(x^2)$. Then $\Omega_{A/k} \cong \frac{A}{xA}dx$, yet there cannot be any injection from $k[t]$ to $A$ as $A$ has Krull dimension zero.

    \item \Cref{unramified implies integral} need not hold when $k$ has positive characteristic. For example, take $\iota$ to be the inclusion $k[x] \rightarrow k[x^{1/p}]$. Then $d(i(x))=0$. 
\end{enumerate}

\end{remark}
\begin{corollary}\label{unramified implies integral corollary}
Let $k$ be a field of characteristic zero, $A$ be a $k$-algebra -- not necessarily noetherian. If $\Omega_{A/k}=0$, then $A$ is integral over $k$.
\end{corollary}

\begin{proof}
Contrary to the assertion of \Cref{unramified implies integral corollary}, assume that for $a \in A$ the $k$-algebra map from the polynomial ring $k[t]$ to $A$ sending $t$ to $a$ is injective. Now \Cref{unramified implies integral} implies that $da \in \Omega_{A/k}$ is nonzero, contradicting our hypothesis $\Omega_{A/k}=0$.
\end{proof}
The next results partially answers Bhatt's question (\Cref{main question}).

\begin{theorem}
\label{formally unramified plus reduced implies ind-etale}
Let $k$ be a field of characteristic zero and $A$ is a reduced $k$-algebra. If $\Omega_{A/k}=0$, then $A$ is ind-\'etale.
\end{theorem}
\begin{proof}
We shall show that any finitely generated $k$-subalgebra of $A$ is \'etale over $k$; this will prove \Cref{formally unramified plus reduced implies ind-etale} since $A$ is the directed union of all finitely generated $k$-subalgebras.\\
Fix a finitely generated $k$-subalgebra $B$ of $A$. The ring $B$ is integral over $k$ as $A$ is integral over $k$ by \Cref{unramified implies integral corollary}. Therefore $B$ has Krull dimension zero. Hence every minimal prime of $B$ is maximal. Since $B$ is noetherian, $B$ has only finitely many minimal primes and hence $B$ has only finitely many maximal ideals -- say $\mathfrak{m}_1, \ldots, \mathfrak{m}_r$. By the Chinese remainder theorem, we have

\begin{equation}\label{chinese remainder theorem}
\frac{B}{\cap_{i=1}^{r} \mathfrak{m}_i} \cong \frac{B}{\mathfrak{m}_1} \times \ldots \times \frac{B}{\mathfrak{m}_r}.
\end{equation}
Since $A$ is reduced, so is $B$. Hence $\cap _{i=1}^{r}\mathfrak{m}_i= 0$. Thus from \Cref{chinese remainder theorem}, we get that $B \cong \frac{B}{\mathfrak{m}_1} \times \ldots \frac{B}{\mathfrak{m}_r}$. Since $k$ has characteristic zero and $B$ is finite type over $k$, for each $i$, $1 \leq i \leq r$, $B/\mathfrak{m}_i$ is a finite, separable field extension of $k$, so $\Omega_{\frac{B}{\mathfrak{m}_i}/k}=0$; see \cite[Tag090W]{Stacksproject}). Finally we conclude $\Omega_{B/k}=0$, since as abelian groups
$$\displaystyle{\Omega_{B/k} \cong \oplus_{i=1}^{r} \Omega_{\frac{B}{\mathfrak{m}_i}/k},}. $$ Thus $B$ is \'etale over $k,$ as desired.
\end{proof}
\begin{remark}
With additional restrictions on $A$, the reducedness hypothesis on $A$ in \Cref{formally unramified plus reduced implies ind-etale} becomes redundant. For example, when $k$ is a perfect field of any characteristic, if $\Omega_{A/k}=0$ and additionally $A$ is noetherian; or a local ring with maximal ideal $\mathfrak{m}$ such that $\underset{n \in \mathbb{N}}{\cap}\mathfrak{m}^n=0$; or $A$ is an $\mathbb{N}$-graded $k$-algebra with $A_0$ is noetherian, then $A$ is automatically reduced; see \cite[Theorem 3.1, Corollary 3.3, Theorem 3.6]{Alapan} for details.
\end{remark}
\begin{remark}\label{Gabber's counterexample} For any characteristic zero field $k$, Gabber has constructed a $k$-algebra $R_{\infty}$ such that $\Omega_{R_\infty /k}=0$, but $R_{\infty}$ is not reduced. The idea is to first construct a direct system $\{R_i \ | \ i \in \mathbb{N} \}$, of finite dimensional local $k$-algebras such that the maps $R_i \rightarrow R_{i+1}$ are injective and the induced maps $\Omega_{R_i/k} \rightarrow \Omega_{R_{i+1}/k}$ are all zero maps. Then $R_\infty$ is taken to be the union of all $R_i$'s. See \cite[Theorem 2.2]{Alapan} for the details of Gabber's construction. 

\end{remark}

\section{Application to Hochschild homology}\label{HH}
 We give an application of \Cref{formally unramified plus reduced implies ind-etale} in Hochschild homology. We begin by giving a minimal review of Hochschild homology here.

\begin{definition}Let $A$ be a commutative ring over a field $k$. Then the $n$-th Hochschild homology $\mathrm{HH}_n(A/k)$ is defined to be $\mathrm{Tor}_n ^{A \otimes_k A}(A,A).$ 
\end{definition}

\begin{remark}
Note that Hochschild homology can be defined for any associative $k$-algebra which is not necessarily commutative. If we denote $A^{\circ}$ to denote the opposite algebra of $A,$ one can in general define $\mathrm{HH}_n (A) := \mathrm{Tor}^{A \otimes_k A^\circ}_n (A,A).$ Thus, $\mathrm{HH}_n(A)$ is really a ``noncommutative invariant" of $A,$ even if $A$ is a commutative algebra.
\end{remark}{}

\begin{remark}There is an explicit chain complex which can be used to compute Hochschild homology groups in general. It is given by 

$$\cdots \to A \otimes_k A \otimes_k A \to  A\otimes_k A \to A \to 0,  $$

where $A$ lives in degree zero. The differentials $d : A^{\otimes_k {n+1}} \mapsto A^{\otimes_k n}$ are given by $$a_0 \otimes \cdots \otimes a_n \to a_0 a_1 \otimes \cdots \otimes a_n - a_0 \otimes a_1 a_2 \otimes \cdots \otimes a_n + \cdots +(-1)^{n}a_0 \otimes \cdots \otimes a_{n-1} a_n + (-1)^{n+1}a_n a_0 \otimes \cdots \otimes a_{n-1}.$$ The complex described above can be viewed as an object in the derived category of $A$ denoted as $D(A),$ where it is quasi-isomorphic to $A \otimes ^L _{A \otimes_k A^\circ} A.$ This object will be denoted by $\mathrm{HH}(A/k) \in D(A).$

\end{remark}{}
We recall an important result about the object $\mathrm{HH}(A/k).$ The result is phrased using the language of filtered objects in derived categories and we refer the reader to \cite{BMS2} for the necessary definitions. The proposition below is obtained by left Kan extending the Postnikov filtration from the smooth case. 

\begin{proposition}(Hochschild–Kostant–Rosenberg (HKR) filtration)\label{HKR} Let $A$ be a commutative $k$-algebra as before. Then $\mathrm{HH}(A/k)$ -- viewed as an object of (the stable $\infty$-category) $D(A)$ admits a natural, complete,
descending $\mathbb N$-indexed filtration, whose $i$-th graded piece is isomorphic to $\mathbb \wedge^i \mathbb{L} _{A/k}[i]$ for $i \ge 0.$
\end{proposition}{}

\begin{proof}
 See \cite[Proposition 2.28]{Morrow} and \cite[Section 2.2]{BMS2}.
\end{proof}{}

\begin{proposition}\label{hhhh}Let $k$ be a field of characteristic $0$ and $A$ be a reduced commutative $k$-algebra. If $\mathrm{HH}_1(A/k)= 0,$ then $\mathrm{HH}_i(A/k)= 0$ for all $i \ge 1.$

\end{proposition}

\begin{proof}We note that $\mathrm{HH}_1(A/k) = \mathrm{Tor}_1^{A \otimes_k A} (A,A) \simeq \Omega_{A/k}$. Therefore, our hypothesis implies that $\Omega_{A/k}= 0.$ Since $A$ is reduced, it follows from \Cref{formally unramified plus reduced implies ind-etale} that $A$ is in fact ind-\'etale and therefore $\mathbb L_{A/k}$ is exact. So $\mathbb{ L}_{A/k}$ is isomorphic to $0$ when viewed as an object of $D(A)$. Let $\text{Fil}^n_{\text{HKR}}(\mathrm{HH}(A/k))$ denote the HKR filtration on $\mathrm{HH}(A/k).$ Since the $i$-th graded piece for the HKR filtration is zero for $i \ge 1$ by \Cref{HKR},, we see that 
\begin{equation}\label{HKRisom}
    \text{Fil}^n_{\text{HKR}}(\mathrm{HH}(A/k)) \simeq \text{Fil}^1_{\text{HKR}}(\mathrm{HH}(A/k)) 
\end{equation}{}for $n \ge 1$. Thus, we have an exact triangle 

$$\text{Fil}^1_{\text{HKR}}(\mathrm{HH}(A/k)) \to \mathrm{HH}(A/k) \to \wedge^0 \mathbb{L}_{A/k}[0] = A[0].$$
Using the fact that the HKR filtration is complete, we argue that $\mathrm{HH}(A/k) \simeq A[0]$; see for e.g., \cite[Definition~5.1]{BMS2} for the definition of a filtered object in the derived category being complete. Indeed, from the completeness of the HKR filtration and \Cref{HKRisom}, it follows that $$0 \simeq R \varprojlim_{n}\text{Fil}^n_{\text{HKR}}(\mathrm{HH}(A/k)) \simeq \text{Fil}^1_{\text{HKR}}(\mathrm{HH}(A/k)).$$ However, by the exact triangle above, that implies that $\mathrm{HH}(A/k) \simeq A[0].$ This finishes the proof.
\end{proof}{}

\begin{eg}\label{exampleBen} \Cref{hhhh} is false if we do not assume the ring to be commutative. A natural source of counterexamples arise from the theory of differential operators. For $n\ge 1$, let $A_n$ denote the $n$-th Weyl algebra over a field $k$ of characteristic $0$; one can also think of $A_n$ as the ring of differential operators of the polynomial ring in $n$ variables over $k$. Concretely, $A_n$ is an associative unital algebra over $k$ generated by $x_1, \ldots, x_n$ and $\partial^1, \ldots, \partial^n$ modulo the relations $x_i x_j = x_j x_i,$ $\partial_i \partial_j = \partial_j \partial_i$ and $\partial_i x_j -x_j \partial_i = \delta_{ij},$ where $\delta_{ij}$ is the Kronecker delta symbol.
\vspace{2mm}

There is a natural increasing and multiplicative filtration on $A_n$ called the order filtration. Since $k$ has characteristic $0$, the associated graded algebra of $A_n$ under the order filtration is a commutative polynomial algebra in $2n$ variables. This implies that $A_n$ is a \textit{reduced} noncommutative $k$-algebra. We note that

$$
\mathrm{HH}_i (A_n) =
\begin{cases}
k & \text{if}\, i = 2n , \\
0 & \text{otherwise};
\end{cases}
$$
see \cite[section 3.1]{Richard} or \cite[section 5]{Sridharan}. This gives a very natural counterexample to \Cref{hhhh} if the ring is not assumed to be commutative. Note that $A_n$ is even an ``almost commutative ring" in the sense of filtered rings.
\end{eg}

\section{Application to $K$-regularity}\label{Kreg}
We begin by very briefly recalling the definition of the higher $K$-groups. For any associative and unital ring $A,$ one can define the nonconnective $K$-theory spectrum $K(A)$ \cite{tttt}, \cite{kbook}. One defines the $K$-groups of $A,$ denoted by $K_n(A)$ for $n \in \mathbb{Z}$ to be the $n$-th homotopy group of the spectrum $K(A),$ i.e.,

$$K_n(A) := \pi_n (K(A)).$$

\begin{remark}\label{elementary}
Let us give more elementary descriptions of some of the $K$-groups that are of relevance to us. We note that $K_0(A)$ is the Grothendieck group of $A$. which is obtained by group completing the monoid of finitely generated projective $A$-modules.

Now we explicitly describe $K_1(A)$; see \cite[Chapter~III, Section~1]{kbook}. For a ring $A,$ note that we have a sequence of group inclusions

$$\mathrm{GL}_1(A) \hookrightarrow \mathrm{GL}_2(A) \hookrightarrow \ldots \hookrightarrow \mathrm{GL}_n(A) \hookrightarrow \ldots.$$
where the inclusion $\mathrm{GL}_n(A) \hookrightarrow \mathrm{GL}_{n+1}(A)$ takes a matrix $M$ to
$\begin{bmatrix}
1 & 0\\
0 & M
\end{bmatrix}$.
Let us denote the group obtained by taking union of the above sequence of inclusions by $\mathrm{GL}(A).$ Let $[\mathrm{GL}(A), \mathrm{GL}(A)]$ denote the derived subgroup, i.e., the subgroup generated by the commutators. Then one has
$$K_1(A) = \mathrm{GL}(A)/ [\mathrm{GL}(A), \mathrm{GL}(A)].$$

The negative $K$-groups can also be described explicitly, in an inductive fashion, using an earlier construction of Bass. For $n<0,$ one has

$$K_n(A) = \mathrm{Coker} \left(K_{n+1}(A[t]) \times K_{n+1}(A[t^{-1}]) \to K_{n+1}(A[t, t^{-1}])\right).$$

The above description can be obtained by covering $\mathbb{P}^1_{A}$ by the two standard affine opens $\text{Spec}\, A[t]$ and $\text{Spec}\, A[t^{-1}]$ and using a Mayer--Vietoris sequence argument (see \cite[Theorem~6.1]{tttt}).

\end{remark}

\begin{definition}\label{teach} A commutative $k$-algebra $A$ is defined to be \textit{$K_n$-regular} if the natural map 

$$K_n(A) \to K_n(A[x_1, \ldots, x_r]) $$ is an isomorphism for all $r \ge 0.$

\end{definition}{}

In \cite[Question 2.9]{weibel}, Weibel asked the following questions.
\begin{question}[Weibel]
Let $R$ be a commutative noetherian ring of Krull dimension $d.$  

\begin{enumerate}
    \item Is $K_n(R) = 0$ for $n < -d$?
    \item Does $R$ happen to be $K_n$-regular for $n \le -d$?
\end{enumerate}{}
\end{question}{}
In \cite{weibel}, Weibel also answered the question when $d=0$ and $1.$ This question was answered in \cite{weibel2} by Cortiñas, Haesemeyer, Schlichting and Weibel for finite type algebras over a field of characteristic zero. The question was completely answered by Kerz, Strunk, and Tamme in \cite{kerz1}. See also \cite{kerz2}. Note that when $R$ is a regular noetherian ring, then all the negative $K$-groups of $R$ vanish \cite{bas}.

It was proven by Quillen in \cite{qln} that a regular noetherian ring is $K_n$-regular for all integers $n.$ The following was conjectured (and proven in dimensions $\le 1$) by Vorst in \cite{vorst}, which predicts the converse. 

\begin{conjecture}[Vorst]\label{vorst}
If $R$ is a commutative ring of dimension $d,$ essentially of finite type over a
field $k$, then $K_{d+1}$-regularity implies regularity.
\end{conjecture}
When $R$ is essentially of finite type over a field $k$ of characteristic zero, Cortiñas, Haesemeyer, and Weibel proved that the above conjecture holds \cite{weibel3}. Positive characteristic variants have been studied by Geisser and Hesselholt in \cite{hess} and
Kerz, Strunk, and Tamme in \cite{kerz3}.

We point out that all the results above makes certain finiteness assumptions. However, as Vorst mentions in \cite{vorst}, it is not clear if the finiteness assumptions are necessary in \Cref{vorst}. As a consequence of \Cref{formally unramified plus reduced implies ind-etale}, we will deduce some instances of Weibel's question (see \Cref{weib}) and Vorst's conjecture (see \Cref{vorstcon}) without any finiteness or even noetherian assumptions.

\begin{proposition}\label{weib} Let $A$ be a commutative $k$-algebra over a field $k$ of characteristic zero such that the module of $(d+1)$-forms $\wedge^{d+1}\Omega_{A/k}=0.$ Then 

\begin{enumerate}
    \item $K_n(A) = 0$ for $n<-d.$
    \item $A$ is $K_n$-regular for all $n \le -d.$ 
\end{enumerate}

\end{proposition}{}

\begin{proof}
By \Cref{unramified implies integral}, the $k$-transcendence cardinality of $A$ is $\le d.$ It then follows from (2), \Cref{p: transcendence degree} that any finite type $k$-subalgebra of $A$ must have Krull dimension $\le d.$ Therefore, by using \cite[Theorem~6.2]{weibel2} and taking filtered colimits over all finite type $k$-subalgebras of $A$, we obtain the desired conclusion.
\end{proof}


\begin{proposition}\label{vorstcon}
Let $A$ be a commutative $k$-algebra over a field $k$ of characteristic zero such that the module of K\"ahler differentials $\Omega_{A/k}=0.$ If $A$ is $K_1$-regular, then $A$ is ind-\'etale. Moreover, $A$ is $K_n$-regular for all integers $n.$
\end{proposition}{}

\begin{proof}
One observes that $A$ being $K_1$-regular implies that $A$ is reduced. In order to see this, we note some well-known general constructions. For any commutative ring $R,$ taking determinant induces a natural group homomorphism $\text{det}: \mathrm{GL}(R) \to R^{\times},$ which factors to give a map $$\text{det}: K_1(R) \to R^{\times}.$$ Here $R^\times$ is the abelian group of units of $R$. Note that there is also a natural map $R^\times = \mathrm{GL}_1(R) \to K_1(A)$ which admits a section provided by $\text{det}: K_1(R) \to R^{\times}.$
\vspace{2mm}
Coming back to our situation, the $K_1$-regularity of $A$ in particular implies that the map $K_1(A) \to K_1(A[t])$ induced by the natural inclusion $A \hookrightarrow A[t]$ is an isomorphism. We have the following commutative diagram where the vertical arrows are given by the `det' maps and the horizontal maps are induced by the inclusion $A \hookrightarrow A[t].$
\begin{center}
\begin{tikzcd}
A^{\times} \arrow[rr]       &  & {(A[t])^{\times}}     \\
K_1(A) \arrow[u] \arrow[rr] &  & {K_1(A[t])} \arrow[u]
\end{tikzcd}
\end{center}{}
Since the vertical maps and the bottom horizontal maps are surjective, the inclusion $A^\times \hookrightarrow (A[t])^\times$ is also surjective. If there were a nonzero nilpotent element $a \in A,$ the element $1+ at \in (A[t])^\times$ would not come from $A^\times.$ Thus we conclude that $A$ is reduced. \Cref{formally unramified plus reduced implies ind-etale} now implies that $A$ is ind-\'etale. For the last part of the proposition, we again note that the $K$-groups commute with taking direct limits and \'etale algebras are $K_n$-regular for all $n,$ which yields the claim.
\end{proof}
Note that \Cref{vorstcon} provides a criteria for an algebra being ind-\'etale in terms of certain condition on the differential forms and $K_1$-regularity. It seems to be an interesting question to find higher dimensional generalizations of this proposition, which would give a criteria for ind-smoothness. Motivated by \cref{vorstcon}, we formulate the following question which we do not know how to answer.
\begin{question}\label{q: non-noetherian vorst}
    Let $k$ be a field of characteristic zero and $A$ be a $k$-algebra such that $\wedge^{d+1}\Omega_{A/k}= 0$. Suppose that $A$ is $K_n$ regular for all $n.$ Is $A$ necessarily a direct limit of smooth $k$-algebras?
\end{question}
The above question imposes $K_n$-regularity condition on the algebra $A,$ which is motivated by Vorst's conjecture. However, we point out that there is a difference between the formulation of classical Vorst conjecture and \Cref{q: non-noetherian vorst} or \Cref{vorstcon}. In the classical version (see \Cref{vorst}) the $K$-regularity assumption and the conjectured regularity, both involve absolute notions such as $K$-groups and regular rings; the essentially finite type hypothesis serves as an assumption making other techniques (such as the crucial usage of the \textit{cdh} topology) applicable in the problem. But, in \Cref{q: non-noetherian vorst}, the $K$-regularity assumption involves absolute notions whereas the desired conclusion in the question, namely the ind-$k$-smoothness is a relative notion (as it refers to the base $k$); the $\wedge^{d+1}\Omega_{A/k}= 0$ assumption however is again a relative assumption. The latter ensures for example, that all finite type $k$-subalgebras of $A$ have dimension at most $d$ (by \Cref{p: transcendence degree}, (2) and (7)).
\printbibliography[title={References}]

\end{document}